\documentclass[12pt]{article}
\usepackage[T1]{fontenc}
\usepackage{iwona,amsmath,amsthm,amsfonts,amssymb,microtype,thmtools,underscore,mathtools,anyfontsize} 
\usepackage[usenames,dvipsnames,svgnames,table]{xcolor}
\usepackage[unicode=true]{hyperref}
\hypersetup{%
colorlinks,
breaklinks=true,
linkcolor={black},
citecolor={black},
urlcolor={blue!60!black},
pdftitle={Tree-Partitions with Small Bounded Degree Trees},
pdfauthor={Distel--Wood}
}
\usepackage[noabbrev,capitalise]{cleveref}
\crefname{lem}{Lemma}{Lemmas}
\crefname{thm}{Theorem}{Theorems}
\crefname{cor}{Corollary}{Corollaries}
\crefname{prop}{Proposition}{Propositions}
\crefname{conj}{Conjecture}{Conjectures}
\crefname{rmk}{Remark}{Remarks}
\crefname{openproblem}{Open Problem}{Open Problems}
\crefformat{equation}{(#2#1#3)}
\Crefformat{equation}{Equation #2(#1)#3}
\usepackage[shortlabels]{enumitem}
\setlist[itemize]{topsep=0ex,itemsep=0ex,parsep=0.4ex}
\setlist[enumerate]{topsep=0ex,itemsep=0ex,parsep=0.4ex}
\usepackage[longnamesfirst,numbers,sort&compress]{natbib}
\makeatletter
\def\NAT@spacechar{~}
\makeatother
\usepackage[margin=30mm]{geometry}
\renewcommand{\baselinestretch}{1.1}
\allowdisplaybreaks
\newcommand{\defn}[1]{\textcolor{Maroon}{\emph{#1}}}

\newcommand{\NN}{\mathbb{N}}

\renewcommand{\geq}{\geqslant}
\renewcommand{\leq}{\leqslant}

\DeclareMathOperator{\tw}{tw}
\DeclareMathOperator{\tpw}{tpw}

\renewcommand{\thefootnote}{\fnsymbol{footnote}}
\allowdisplaybreaks
\theoremstyle{plain}
\newtheorem{thm}{Theorem}
\newtheorem{lem}[thm]{Lemma}
\newtheorem{cor}[thm]{Corollary}

\theoremstyle{definition}

\begin{document}

\title{\bf\Large Tree-Partitions with Small Bounded Degree Trees\footnote{\today. This work was initiated at the \href{https://www.matrix-inst.org.au/events/structural-graph-theory-downunder-ll/}{Structural Graph Theory Downunder II} workshop at the Mathematical Research Institute MATRIX (March 2022). A preliminary version of this paper appears in the 
\href{https://www.matrix-inst.org.au/2021-matrix-annals/}{2021-22 MATRIX Annals}.}}
\author{Marc Distel\,\footnotemark[3] \qquad David~R.~Wood\,\footnotemark[3]}

\footnotetext[3]{School of Mathematics, Monash University, Melbourne, Australia (\texttt{\{marc.distel,david.wood\}@monash.edu}). Research of D.W.\ supported by the Australian Research Council. Research of M.D.\ supported by an Australian Government Research Training Program Scholarship.}

\date{}
\maketitle

\begin{abstract}
A \defn{tree-partition} of a graph $G$ is a partition of $V(G)$ such that identifying the vertices in each part gives a tree. It is known that every graph with treewidth $k$ and maximum degree $\Delta$ has a tree-partition with parts of size $O(k\Delta)$. We prove the same result with the extra property that the underlying tree has maximum degree $O(\Delta)$ and $O(|V(G)|/k)$ vertices. 
\end{abstract}


\renewcommand{\thefootnote}{\arabic{footnote}}

\section{Introduction}

For a graph~$G$ and a tree~$T$, a \defn{$T$-partition} of~$G$ is a partition~${(V_x \colon x\in V(T))}$ of~${V(G)}$ indexed by the nodes of~$T$, such that for every edge~${vw}$ of~$G$, if~${v \in V_x}$ and~${w \in V_y}$, then~${x = y}$ or~${xy \in E(T)}$. The \defn{width} of a $T$-partition is~${\max\{ |{V_x}| \colon x \in V(T)\}}$. The \defn{tree-partition-width}\footnote{Tree-partition-width has also been called \defn{strong treewidth} \citep{BodEng-JAlg97,Seese85}.} of a graph $G$ is the minimum width of a tree-partition of $G$. 

Tree-partitions were independently introduced by \citet{Seese85} and \citet{Halin91}, and have since been widely investigated \citep{Bodlaender-DMTCS99,BodEng-JAlg97, DO95,DO96,Edenbrandt86, Wood06,Wood09,BGJ22}. 
Applications of tree-partitions include 
graph drawing~\citep{CDMW08,GLM05,DMW05,DSW07,WT07}, 
graphs of linear growth~\citep{CDGHHHMW}, 
nonrepetitive graph colouring~\citep{BW08}, 
clustered graph colouring~\citep{ADOV03,LO18}, 
monadic second-order logic~\citep{KuskeLohrey05}, 
network emulations~\citep{Bodlaender-IPL88, Bodlaender-IC90, BvL-IC86, FF82}, statistical learning theory~\citep{ZA22}, 
size Ramsey numbers~\citep{DKCPS,KLWY21}, 
and the edge-{E}rd{\H{o}}s-{P}{\'o}sa property~\citep {RT17,GKRT16,CRST18}. Tree-partitions are also related to graph product structure theory since a graph $G$ has a $T$-partition of width at most $k$ if and only if $G$ is isomorphic to a subgraph of $T\boxtimes K_k$ for some tree $T$; see~\citep{UTW,DJMMW,DEMWW22} for example.

Bounded tree-partition-width implies bounded treewidth\footnote{A \defn{tree-decomposition} of a graph $G$ is a collection $(B_x\subseteq V(G):x\in V(T))$ of subsets of $V(G)$ (called \defn{bags}) indexed by the nodes of a tree $T$, such that: (a) for every edge $uv\in E(G)$, some bag $B_x$ contains both $u$ and $v$; and (b) for every vertex $v\in V(G)$, the set $\{x\in V(T):v\in B_x\}$ induces a non-empty subtree of $T$. The \defn{width} of a tree-decomposition is the size of the largest bag, minus $1$. 
The \defn{treewidth} of a graph $G$, denoted by \defn{$\tw(G)$}, is the minimum width of a tree-decomposition of $G$. Treewidth is the standard measure of how similar a graph is to a tree. Indeed, a connected graph has treewidth 1 if and only if it is a tree. Treewidth is of fundamental importance in structural and algorithmic graph theory; see \citep{Reed03,HW17,Bodlaender98} for surveys.}, as noted by \citet{Seese85}. In particular, for every graph~$G$,
\[\tw(G) \leq 2\tpw(G)-1.\]
Of course, ${\tw(T) = \tpw(T) = 1}$ for every tree~$T$. 
But in general, $\tpw(G)$ can be much larger than~$\tw(G)$. 
For example, fan graphs on~$n$ vertices have treewidth~2 and tree-partition-width~$\Omega(\sqrt{n})$. 
On the other hand, the referee of \citep{DO95} showed that if the maximum degree and treewidth are both bounded, then so is the tree-partition-width, which is one of the most useful results about tree-partitions. A graph $G$ is \defn{trivial} if $E(G)=\emptyset$. Let \defn{$\Delta(G)$} be the maximum degree of $G$. 
 
\begin{thm}[\cite{DO95}]
\label{TreeProduct}
For any non-trivial graph $G$,
$$\tpw(G) \leq 24 (\tw(G)+1)\Delta(G).$$
\end{thm}

\cref{TreeProduct} is stated in \cite{DO95} with ``$\tw(G)$'' instead of ``$\tw(G)+1$'', but a close inspection of the proof shows that ``$\tw(G)+1$'' is needed. \citet{Wood09} showed that \cref{TreeProduct} is best possible up to the multiplicative constant, and also improved the constant $24$ to $9+6\sqrt{2}\approx 17.48$.

This paper proves similar results to \cref{TreeProduct} where the tree $T$ indexing the tree-partition has two extra properties. 

The first property is the maximum degree of $T$. Consider a tree-partition $(B_x:x\in V(T))$ of a graph $G$ with width $k$. For each node $x\in V(T)$, there are at most $\sum_{v\in B_x}\deg(v)$ edges between $B_x$ and $G-B_x$. Thus we may assume that $\deg_T(x)\leq |B_x|\Delta(G)\leq k\Delta(G)$, otherwise delete an `unused' edge of $T$ and add an edge to $T$ between leaf vertices of the resulting component subtrees. It follows that if $\tpw(G)\leq k$ then $G$ has a $T$-partition of width at most $k$ for some tree $T$ with maximum degree at most $\max\{k\Delta(G),2\}$. By \cref{TreeProduct}, every graph $G$ has a $T$-partition of width at most $24 (\tw(G)+1)\Delta(G)$ for some tree $T$ with maximum degree at most $24 (\tw(G)+1)\Delta(G)^2$. This fact has been used in several applications of \cref{TreeProduct} (see \citep{CDMW08,DSW07} for example). 

The second property is the number of vertices in $T$. This property is motivated by results about size-Ramsey number by \citet{DKCPS}, where tree-partitions of $n$-vertex graphs with $\tw(G)\in O(\sqrt{n})$ play a critical role, and it is essential that $\Delta(T)$ is independent of $\tw(G)$ and $|V(T)| \ll |V(G)|$. 

The following theorem improves the above-mentioned upper bound on $\Delta(T)$ so that it is independent of $\tw(G)$, and provides a non-trivial upper bound on $|V(T)|$. Let $\NN=\{1,2,\dots\}$. 

\begin{thm}
\label{ImprovedTreeProduct}
For any integers $k,d\in\NN$, every non-trivial graph $G$ with $\tw(G)\leq k-1$ and $\Delta(G)\leq d$ has a $T$-partition of width at most \(18kd\), for some tree $T$ with $\Delta(T)\leq 6d$ and $|V(T)|\leq\max\{\frac{|V(G)|}{2k},1\}$.
\end{thm}

\cref{ImprovedTreeProduct} enables a $\tw(G)\Delta(G)^2$ term to be replaced by a $\Delta(G)$ term in various  results~\citep{CDMW08,DSW07}. And since $\Delta(T)$ is independent of $\tw(G)$, and $|V(T)|\ll |V(G)|$ when $\tw(G)$ is unbounded, \cref{ImprovedTreeProduct} is suitable for the applications to size-Ramsey numbers in  \citep{DKCPS}. Indeed, \cref{ImprovedTreeProduct} improves upon a similar result of \citet[Lemma~2.1]{DKCPS} which has an extra $O(\log n)$ factor in the width.

Here we give an example of \cref{ImprovedTreeProduct}. \citet{AST90} proved that every $K_t$-minor-free graph $G$ has treewidth at most  $t^{3/2}|V(G)|^{1/2}-1$. \cref{ImprovedTreeProduct} thus implies:

\begin{cor}
\label{MinorClosedClass}
Every $K_t$-minor-free graph with maximum degree $\Delta$ has a $T$-partition of width at most $18t^{3/2}\Delta|V(G)|^{1/2}$, for some tree $T$ with $\Delta(T)\leq 6\Delta$ and $|V(T)|\leq |V(G)|^{1/2}/2t^{3/2}$.
\end{cor}

As mentioned above, \citet{Wood09} improved the constant $24$ to $9+6\sqrt{2}$ in \cref{TreeProduct}. By tweaking the constants in the proof of \cref{ImprovedTreeProduct}, we match this constant with a small increase in the bound on $\Delta(T)$; see \cref{tweaking}.
We choose to first present the proof with integer coefficients for ease of understanding.

Our final result shows that the linear upper bound on $\Delta(T)$ in \cref{ImprovedTreeProduct} is best possible even for trees.

\begin{restatable}{prop}{LowerBound}
\label{LowerBound}
For any integer $\Delta\geq 3$ there exist $\alpha>0$ such that there are infinitely many trees $X$ with maximum degree $\Delta$ such that for every tree $T$ with maximum degree less than $\Delta$, every $T$-partition of $X$ has width at least $|V(X)|^\alpha$. Moreover, if $\Delta=3$ then $\alpha$ can be taken to be arbitrarily close to $1$. 
\end{restatable}

\section{Main Proofs}

The proof of \cref{ImprovedTreeProduct} closely follows the proof of \cref{TreeProduct}, except that we pay attention to $\Delta(T)$ and $|V(T)|$. 
The next lemma is the heart of our proof. 

\begin{lem}
\label{heart}
For $k,d\in\NN$, for any graph $G$ with $\tw(G)\leq k-1$ and $\Delta(G)\leq d$, for any set $S\subseteq V(G)$ with $4 k\leq|S| \leq 12 kd$, there exists a tree-partition $(B_x:x\in V(T))$ of $G$ such that:
\begin{itemize}
    \item $|B_x|\leq 18kd $ for each $x\in V(T)$,
    \item $\Delta(T)\leq 6d$, 
    \item $|V(T)|\leq \frac{|V(G)|}{2k}$.
\end{itemize}
Moreover, there exists $z\in V(T)$ such that:
\begin{itemize}
    \item $S\subseteq B_z$, 
    \item $|B_z|\leq \frac32|S|-2k$,
    \item $\deg_T(z)\leq \frac{|S|}{2k} - 1$.
\end{itemize}
\end{lem}

\begin{proof}
We proceed by induction on $|V(G)$|.

\textbf{Case 1.} $|V(G-S)|\leq 18kd$: Let $T$ be the tree with $V(T)=\{y,z\}$ and $E(T)=\{yz\}$. Note that $\Delta(T)=1\leq 6 d$ and $|V(T)|=2\leq \frac{|S|}{2k} \leq \frac{|V(G)|}{2k}$ and $\deg_T(z)=1\leq \frac{|S|}{2k}-1$. Let $B_z:=S$ and $B_y:=V(G-S)$. Thus $|B_z|=|S|\leq \frac32 |S|-2k\leq 18kd$ and $|B_y|\leq |V(G-S)|\leq 18kd$. Hence $(B_x:x\in V(T))$ is the desired tree-partition of $G$. 

Now assume that $|V(G-S)|\geq 18kd$.

\textbf{Case 2.} $4k \leq |S|\leq 12 k$: Let $S':=\bigcup\{ N_G(v)\setminus S: v\in S\}$. Thus $|S'|\leq d |S|\leq 12kd$. If $|S'|< 4 k$ then add $4k-|S'|$ vertices from $V(G-S-S')$ to $S'$, so that $|S'|=4k$. This is well-defined since 
$|V(G-S)| \geq 18kd \geq 4k$, implying $|V(G-S-S')| \geq 4k-|S'|$.
By induction, there exists a tree-partition $(B_x:x\in V(T'))$ of $G-S$ 
such that:
\begin{itemize}
    \item $|B_x|\leq 18kd $ for each $x\in V(T'))$,
    \item $\Delta(T')\leq 6d$, 
    \item $|V(T')|\leq \frac{|V(G-S)|}{2k}$.
\end{itemize}
Moreover, there exists $z'\in V(T')$ such that:
\begin{itemize}
    \item $S'\subseteq B_{z'}$, 
    \item $|B_{z'}|\leq \frac32 |S'|-2k \leq 18kd -2k$,
    \item $\deg_{T'}(z')\leq \frac{|S'|}{2k} - 1 \leq 6d - 1$.
\end{itemize}
Let $T$ be the tree obtained from $T'$ by adding one new node $z$ adjacent to $z'$. Let $B_z:=S$. So $(B_x:x\in V(T))$ is a tree-partition of $G$ with width at most $\max\{18kd,|S|\}\leq\max\{18kd,12k\}=18kd$. By construction, $\deg_T(z)=1 \leq \frac{|S|}{2k} - 1$ and $\deg_{T}(z') = \deg_{T'}(z')+1\leq (6d-1) + 1 = 6d$. Every other vertex in $T$ has the same degree as in $T'$. Hence $\Delta(T)\leq 6d$, as desired. 
Also $|V(T)|=|V(T')|+1\leq \frac{|V(G-S)|}{2k} +1 <\frac{|V(G)|}{2k}$ since $|S|\geq 4k$. Finally, $S=B_z$ and $|B_z|=|S| \leq \frac32 |S|-2k$.

\textbf{Case 3.} $12 k \leq |S|\leq 12kd$: By the separator lemma of \citet[(2.6)]{RS-II}, there are induced subgraphs $G_1$ and $G_2$ of $G$ with $G_1\cup G_2=G$ and $|V(G_1\cap G_2)|\leq k$, where $|S\cap V(G_i)|\leq \frac23 |S|$ for each $i\in\{1,2\}$. Let $S_i := (S\cap V(G_i))\cup V(G_1\cap G_2)$ for each $i\in\{1,2\}$.

We now bound $|S_i|$. For a lower bound, since $|S\cap V(G_1)|\leq \frac23 |S|$, we have $|S_2|\geq |S\setminus V(G_1)|\geq \frac13 |S| \geq 4k $. By symmetry, $|S_1|\geq  4k $. For an upper bound, $|S_i|\leq\frac23 |S| + k \leq 8kd + k \leq 12kd$. Also note that $|S_1|+|S_2|\leq |S|+2k$.

We have shown that $4k \leq |S_i|\leq 12kd$ for each $i\in\{1,2\}$. Thus we may apply induction to $G_i$ with $S_i$ the specified set. Hence there exists a tree-partition $(B^i_x:x\in V(T_i))$ of $G_i$ 
such that:
\begin{itemize}
    \item $|B^i_x|\leq 18kd $ for each $x\in V(T_i))$,
    \item $\Delta(T_i)\leq 6d$, 
    \item $|V(T_i)|\leq \frac{|V(G_i)|}{2k}$.
\end{itemize}
Moreover, there exists $z_i\in V(T_i)$ such that:
\begin{itemize}
    \item $S_i\subseteq B_{z_i}$, 
    \item $|B_{z_i}|\leq \frac32|S_i|-2k$,
    \item $\deg_{T_i}(z_i)\leq \frac{|S_i|}{2k}-1$.
\end{itemize}
Let $T$ be the tree obtained from the disjoint union of $T_1$ and $T_2$ by identifying $z_1$ and $z_2$ into a vertex $z$. Let $B_z:= B^1_{z_1}\cup B^2_{z_2}$. Let $B_x:= B^i_x$ for each $x\in V(T_i)\setminus\{z_i\}$. Since $G=G_1\cup G_2$ and $V(G_1\cap G_2)\subseteq B^1_{z_1}\cap B^2_{z_2} \subseteq B_z$, we have that $(B_x:x\in V(T))$ is a tree-partition of $G$. 
By construction, $S\subseteq B_z$ and since $V(G_1\cap G_2)\subseteq B^i_{z_i}$ for each $i$, 
\begin{align*}
    |B_z| 
    & \leq |B^1_{z_1}|+|B^2_{z_2}| - |V(G_1\cap G_2)|\\
    & \leq (\tfrac32|S_1|-2k) +  (\tfrac32|S_2|-2k) - |V(G_1\cap G_2)|\\
    & = \tfrac32( |S_1|+ |S_2|) -4k - |V(G_1\cap G_2)|\\
    & \leq \tfrac32( |S| + 2|V(G_1\cap G_2)| ) -4k - |V(G_1\cap G_2)|\\
    & \leq \tfrac32 |S| + 2|V(G_1\cap G_2)| -4k\\
    & \leq \tfrac32 |S| - 2 k\\
    & < 18 kd.
\end{align*}
Every other part has the same size as in the tree-partition of $G_1$ or $G_2$. So this tree-partition of $G$ has width at most $18kd$. 
Note that 
\begin{align*}
 \deg_T(z)   = \deg_{T_1}(z_1) + \deg_{T_2}(z_2)
    & \leq  (\tfrac{|S_1|}{2k}-1) + (\tfrac{|S_2|}{2k}-1)\\
    & =  \tfrac{|S_1|+|S_2|}{2k} -2\\
    & \leq  \tfrac{|S|+2k}{2k} -2\\
    & =  \tfrac{|S|}{2k} -1\\
    & < 6 d. 
\end{align*}
Every other node of $T$ has the same degree as in $T_1$ or $T_2$. 
Thus $\Delta(T) \leq 6d$. 
Finally, 
\begin{align*}
    |V(T)|  =|V(T_1)|+|V(T_2)|-1 
    & \leq \tfrac{|V(G_1)|}{2k}+\tfrac{|V(G_2)|}{2k}-1\\
    & \leq  \tfrac{|V(G)|+k}{2k}-1\\
    & < \tfrac{|V(G)|}{2k}.
\end{align*}
This completes the proof.
\end{proof}

We now prove our main result. 

\begin{proof}[Proof of \cref{ImprovedTreeProduct}]
First suppose that $|V(G)| < 4 k$. Let $T$ be the 1-vertex tree with $V(T)=\{x\}$, and let $B_x:=V(G)$. Then $(B_x:x\in V(T))$ is the desired tree-partition, since $|V(T)|= 1 \leq \max\{\frac{|V(G)|}{2k},1\}$ and $|B_x|=|V(G)|<4 k \leq 18 kd$ and $\Delta(T)=0\leq 6d$. 
Now assume that $|V(G)| \geq 4 k$. The result follows from \cref{heart}, where $S$ is any set of $4k$ vertices in $G$. 
\end{proof}

We now prove the lower bound. For $\Delta,d\in\NN$ with $\Delta\geq 2$, let $X_{\Delta,d}$ be the tree rooted at a vertex $r$ such that every leaf is at distance $d$ from $r$ and every non-leaf vertex has degree $\Delta$. Observe that $X_{\Delta,d}$ has the maximum number of vertices in a tree with maximum degree $\Delta$ and radius $d$, where  
\[ |V(X_{\Delta,d})| 
= 1+\Delta\sum_{i=0}^{d-1}(\Delta-1)^i.\]
Note that $|V(X_{2,d})|= 2d+1$, and
 if $\Delta\geq 3$ then 
\[ (\Delta-1)^d \leq |V(X_{\Delta,d})| 
=1+\tfrac{\Delta}{\Delta-2}\big((\Delta-1)^d-1\big)
\leq 3(\Delta-1)^d.\]

\LowerBound*

\begin{proof}
First suppose that $\Delta\geq 4$. 
Let $d_0\in\NN$ be sufficiently large so that $(\frac{\Delta-1}{\Delta-2})^{d_0}>3$. 
Let $\alpha:= {1- \log_{\Delta-1}(3^{1/d_0}(\Delta-2))}$, which is positive by the choice of $d_0$. 
Let $d\in\NN$ with $d\geq d_0$. 
It follows that 
$(\Delta-1)^{(1-\alpha)d} \geq 3(\Delta-2)^d $. 
Consider any tree-partition $(B_u:u\in V(T))$ of $X_{\Delta,d}$, where $T$ is any tree with maximum degree at most $\Delta-1$. Let $z$ be the vertex of $T$ such that the root $r\in B_z$. Since adjacent vertices in $X_{\Delta,d}$ belong to adjacent parts or the same part in $T$, every vertex in $T$ is at distance at most $d$ from $z$. Thus $T$ has radius at most $d$, and \[|V(T)|\leq|V(X_{\Delta-1,d})| 
  \leq 3(\Delta-2)^d 
  \leq (\Delta-1)^{(1-\alpha)d} \leq
|V(X_{\Delta,d})|^{1-\alpha}  .\]
By the pigeon-hole principle, there is a vertex $u\in V(T)$ such that 
$|B_u|\geq \frac{|V(X_{\Delta,d})|}{|V(T)|}\geq |V(X_{\Delta,d})|^\alpha$.

Now assume that $\Delta=3$. Let $\alpha\in(0,1)$, let $d_0\in\NN$ be sufficiently large so that $2d_0+1 \leq 2^{(1-\alpha)d_0}$, and let $d\geq d_0$. So $2d+1 \leq 2^{(1-\alpha)d}$. Consider any tree-partition $(B_u:u\in V(T))$ of $X_{3,d}$, where $T$ is any tree with maximum degree at most $2$. By the argument above, $T$ has radius at most $d$, implying \[|V(T)|\leq|V(X_{2,d})|=2d+1 \leq 2^{(1-\alpha)d} \leq |V(X_{3,d})|^{1-\alpha}.\]
Again, there is a vertex $u\in V(T)$ such that 
$|B_u|\geq \frac{|V(X_{3,d})|}{|V(T)|}\geq |V(X_{3,d})|^{\alpha}$.
\end{proof}

\section{Tweaking the Constants} 
\label{tweaking}

Consider the following generalisation of \cref{heart}.

\begin{lem}
\label{alpha}
Fix $k,d\in\NN$ and $\alpha\in\mathbb{R}$ with $\alpha>2$. 
For any  graph $G$ with $\tw(G)\leq k-1$ and $\Delta(G)\leq d$,  
for any set $S\subseteq V(G)$ with $\alpha  k\leq|S| \leq 3 \alpha   kd$, there exists a tree-partition $(B_x:x\in V(T))$ of $G$ such that:
\begin{itemize}
    \item $|B_x|\leq \frac{3\alpha(\alpha-1)}{\alpha-2} kd - \frac{\alpha}{\alpha-2} k$ for each $x\in V(T)$,
    \item $\Delta(T)\leq \frac{3 \alpha}{\alpha-2}d + \frac{\alpha-4}{\alpha-2}$, 
    \item $|V(T)|\leq \tfrac{2}{\alpha k}|V(G)|$.
\end{itemize}
Moreover, there exists $z\in V(T)$ such that:
\begin{itemize}
    \item $S\subseteq B_z$, 
    \item $|B_z|\leq \frac{\alpha-1}{\alpha-2} |S|- \frac{\alpha}{\alpha-2} k$,
    \item $\deg_T(z)\leq \frac{1}{(\alpha-2)k} |S| - \frac{2}{\alpha-2}$.
\end{itemize}
\end{lem}

\begin{proof}
We proceed by induction on $|V(G)$|.



\textbf{Case 1.} $|V(G-S)|\leq 
\frac{3\alpha(\alpha-1)}{\alpha-2} kd - \frac{\alpha}{\alpha-2} k$: Let $T$ be the tree with $V(T)=\{y,z\}$ and $E(T)=\{yz\}$. Note that 
$|V(T)|=2\leq \frac{2}{\alpha k}|S| \leq \tfrac{2}{\alpha k}|V(G)|$ and $\Delta(T)=1\leq \frac{3 \alpha}{\alpha-2} d + \frac{\alpha-4}{\alpha-2}$ and $\deg_T(z)=1\leq \frac{1}{(\alpha-2)k}|S|-\frac{2}{\alpha-2}$. Let $B_z:=S$ and $B_y:=V(G-S)$. Thus $|B_z|=|S|\leq \frac{\alpha-1}{\alpha-2} |S|- \frac{\alpha}{\alpha-2} k\leq \frac{3\alpha(\alpha-1)}{\alpha-2} kd - \frac{\alpha}{\alpha-2} k$ and $|B_y|\leq |V(G-S)|\leq \frac{3\alpha(\alpha-1)}{\alpha-2}     kd$. Hence $(B_x:x\in V(T))$ is the desired tree-partition of $G$. 

Now assume that $|V(G-S)|\geq 
\frac{3\alpha(\alpha-1)}{\alpha-2} kd - \frac{\alpha}{\alpha-2} k$.

\textbf{Case 2.} $\alpha k \leq |S|\leq 3 \alpha   k$: Let $S':=\bigcup_{v\in S}N_G(v)\setminus S$. Thus $|S'|\leq d |S|\leq 3 \alpha   kd$. If $|S'|< \alpha  k$ then add $\alpha k-|S'|$ vertices from $V(G-S-S')$ to $S'$, so that $|S'|=\alpha k$. This is well-defined since 
$|V(G-S)| \geq 
\frac{3\alpha(\alpha-1)}{\alpha-2} kd - \frac{\alpha}{\alpha-2} k \geq \alpha k$, implying $|V(G-S-S')| \geq \alpha k-|S'|$.
By induction, there exists a tree-partition $(B_x:x\in V(T'))$ of $G-S$ such that:
\begin{itemize}
    \item $|B_x|\leq \frac{3\alpha(\alpha-1)}{\alpha-2} kd - \frac{\alpha}{\alpha-2} k$ for each $x\in V(T')$,
    \item $\Delta(T')\leq \frac{3 \alpha}{\alpha-2}d + \frac{\alpha-4}{\alpha-2}$, 
    \item $|V(T')|\leq \frac{2}{\alpha k}|V(G-S)|$.
\end{itemize}
Moreover, there exists $z'\in V(T')$ such that:
\begin{itemize}
    \item $S'\subseteq B_{z'}$, 
    \item $|B_{z'}|\leq \frac{\alpha-1}{\alpha-2} |S'|- \frac{\alpha}{\alpha-2} k \leq \frac{3 \alpha (\alpha-1)}{\alpha-2}   kd - \frac{\alpha}{\alpha-2} k$,
    \item $\deg_{T'}(z')\leq \frac{1}{(\alpha-2)k} |S'| - \frac{2}{\alpha-2} \leq  \frac{3 \alpha}{\alpha-2}  d -  \frac{2}{\alpha-2} $.
\end{itemize}
Let $T$ be the tree obtained from $T'$ by adding one new node $z$ adjacent to $z'$. Let $B_z:=S$. So $(B_x:x\in V(T))$ is a tree-partition of $G$ with width at most 
\[\max\{
\tfrac{3\alpha(\alpha-1)}{\alpha-2} kd - \tfrac{\alpha}{\alpha-2} k,|S|\} \leq \max\{
\tfrac{3\alpha(\alpha-1)}{\alpha-2} kd - \tfrac{\alpha}{\alpha-2} k,3 \alpha  k\} = 
\tfrac{3\alpha(\alpha-1)}{\alpha-2} kd - \tfrac{\alpha}{\alpha-2} k.\] By construction, 
\begin{align*}
\deg_T(z) & = 1 \leq \tfrac{1}{(\alpha-2)k} |S| -  \tfrac{2}{\alpha-2} \quad \text{and }\\
\deg_{T}(z')  & = \deg_{T'}(z')+1 \leq \tfrac{3 \alpha}{\alpha-2}   d -  \tfrac{2}{\alpha-2}  + 1  = \tfrac{3 \alpha}{\alpha-2}    d
+ \tfrac{\alpha-4}{\alpha-2}. 
\end{align*}
Every other vertex in $T$ has the same degree as in $T'$. Hence $\Delta(T)\leq \frac{3 \alpha}{\alpha-2}d + \frac{\alpha-4}{\alpha-2}$, as desired. 
Also $|V(T)|=|V(T')|+1\leq \frac{2}{\alpha k}|V(G-S)| +1 < \tfrac{2}{\alpha k}|V(G)|$ since $|S|\geq \alpha k$. 
Finally, $S=B_z$ and $|B_z|=|S| \leq \frac{\alpha-1}{\alpha-2}  |S|- \frac{\alpha}{\alpha-2} k$.

\textbf{Case 3.} $3 \alpha   k \leq |S|\leq 3 \alpha  kd$: By the separator lemma of \citet[(2.6)]{RS-II}, there are induced subgraphs $G_1$ and $G_2$ of $G$ with $G_1\cup G_2=G$ and $|V(G_1\cap G_2)|\leq k$, where $|S\cap V(G_i)|\leq \frac23 |S|$ for each $i\in\{1,2\}$. Let $S_i := (S\cap V(G_i))\cup V(G_1\cap G_2)$ for each $i\in\{1,2\}$.

We now bound $|S_i|$. For a lower bound, since $|S\cap V(G_1)|\leq \frac23 |S|$, we have $|S_2|\geq |S\setminus V(G_1)|\geq \frac13 |S| \geq \frac13 3 \alpha  k \geq \alpha k $. By symmetry, $|S_1|\geq  \alpha k $. For an upper bound, $|S_i|\leq\frac23 |S| + k \leq 2 \alpha  kd + k \leq 3 \alpha  kd$. Also note that $|S_1|+|S_2| \leq |S|+2|V(G_1\cap G_2)| \leq |S|+2k$. We have shown that $\alpha k \leq |S_i|\leq 3 \alpha  kd$ for each $i\in\{1,2\}$. Thus we may apply induction to $G_i$ with $S_i$ the specified set. Hence there exists a tree-partition $(B^i_x:x\in V(T_i))$ of $G_i$ such that:
\begin{itemize}
    \item $|B^i_x|\leq \frac{3\alpha(\alpha-1)}{\alpha-2} kd - \frac{\alpha}{\alpha-2} k$ for each $x\in V(T_i)$,
    \item $\Delta(T_i)\leq \frac{3 \alpha}{\alpha-2}d + \frac{\alpha-4}{\alpha-2} \leq \frac{|V(G_i)|}{2k}$, 
    \item $|V(T_i)|\leq \frac{2}{\alpha k}|V(G_i)|$.
\end{itemize}
Moreover, there exists $z_i\in V(T_i)$ such that:
\begin{itemize}
    \item $S_i\subseteq B_{z_i}$, 
    \item $|B_{z_i}|\leq \frac{\alpha-1}{\alpha-2} |S_i|- \frac{\alpha}{\alpha-2} k$,
    \item $\deg_{T_i}(z_i)\leq \frac{1}{(\alpha-2)k} |S_i| -  \frac{2}{\alpha-2} $.
\end{itemize}
Let $T$ be the tree obtained from the disjoint union of $T_1$ and $T_2$ by identifying $z_1$ and $z_2$ into a vertex $z$. Let $B_z:= B^1_{z_1}\cup B^2_{z_2}$. Let $B_x:= B^i_x$ for each $x\in V(T_i)\setminus\{z_i\}$. Since $G=G_1\cup G_2$ and $V(G_1\cap G_2) = B^1_{z_1}\cap B^2_{z_2} \subseteq B_z$, we have that $(B_x:x\in V(T))$ is a tree-partition of $G$. By construction, $S\subseteq B_z$ and since $V(G_1\cap G_2)\subseteq B^i_{z_i}$ for each $i$, 
\begin{align*}
    |B_z| 
    & \leq |B^1_{z_1}|+|B^2_{z_2}| - |V(G_1\cap G_2)|\\
    & \leq (\tfrac{\alpha-1}{\alpha-2} |S_1|- \tfrac{\alpha}{\alpha-2} k) +  (\tfrac{\alpha-1}{\alpha-2} |S_2|- \tfrac{\alpha}{\alpha-2} k) - |V(G_1\cap G_2)|\\
    & = \tfrac{\alpha-1}{\alpha-2} ( |S_1|+ |S_2|) - \tfrac{2\alpha}{\alpha-2} k - |V(G_1\cap G_2)|\\
    & \leq \tfrac{\alpha-1}{\alpha-2} ( |S| + 2|V(G_1\cap G_2)| ) - \tfrac{2\alpha}{\alpha-2} k - |V(G_1\cap G_2)|\\
    & = \tfrac{\alpha-1}{\alpha-2}  |S|  - \tfrac{2\alpha}{\alpha-2} k + \tfrac{\alpha}{\alpha-2}  |V(G_1\cap G_2)|\\
    & \leq \tfrac{\alpha-1}{\alpha-2}  |S|  - \tfrac{2\alpha}{\alpha-2} k + \tfrac{\alpha}{\alpha-2}  k\\
    & = \tfrac{\alpha-1}{\alpha-2} |S| -  \tfrac{\alpha}{\alpha-2}  k \\
    & \leq \tfrac{3\alpha(\alpha-1)}{\alpha-2}     kd -  \tfrac{\alpha}{\alpha-2}  k .
\end{align*}
Every other part has the same size as in the tree-partition of $G_1$ or $G_2$. So this tree-partition of $G$ has width at most $\frac{3\alpha(\alpha-1)}{\alpha-2} kd - \frac{\alpha}{\alpha-2} k$.  Note that 
\begin{align*}
 \deg_T(z)  & = \deg_{T_1}(z_1) + \deg_{T_2}(z_2)\\
    & \leq  \tfrac{1}{(\alpha-2)k}|S_1|- \tfrac{2}{\alpha-2} + \tfrac{1}{(\alpha-2)k}|S_2|- \tfrac{2}{\alpha-2} \\
    & \leq  \tfrac{1}{(\alpha-2)k}(|S_1|+|S_2|)- \tfrac{4}{\alpha-2} \\
    & \leq  \tfrac{1}{(\alpha-2)k}(|S|+2k) -
    \tfrac{4}{\alpha-2} \\
    & \leq  \tfrac{1}{(\alpha-2)k}|S|- \tfrac{2}{\alpha-2} \\
    & \leq  \tfrac{3 \alpha}{\alpha-2} d- \tfrac{2}{\alpha-2} \\
    & < \tfrac{3 \alpha}{\alpha-2}     d + \tfrac{\alpha-4}{\alpha-2}. 
\end{align*}
Every other node of $T$ has the same degree as in $T_1$ or $T_2$. 
Thus $\Delta(T) \leq \tfrac{3 \alpha}{\alpha-2}    d + \tfrac{\alpha-4}{\alpha-2}$. Finally, 
\begin{align*}
    |V(T)|  =|V(T_1)|+|V(T_2)|-1 
    & \leq \tfrac{2}{\alpha k}|V(G_1)| + \tfrac{2}{\alpha k}|V(G_2)|-1\\
    & \leq  \tfrac{2}{\alpha k}(|V(G)|+k)-1\\
    & < \tfrac{2}{\alpha k}|V(G)|.
\end{align*}
This completes the proof.
\end{proof}

A proof directly analogous to the proof of \cref{ImprovedTreeProduct} using 
\cref{alpha} with $\alpha=4$ instead of \cref{heart} implies the following slight strengthening of \cref{ImprovedTreeProduct}.

\begin{thm}
For $k\in\NN$, every non-trivial graph $G$ with $\tw(G)\leq k-1$ has a $T$-partition of width at most \[2\tw(G)\,(9\Delta(G)-1),\] 
for some tree $T$ with $\Delta(T)\leq 6\Delta(G)$ and $|V(T)|\leq \max\{\frac{|V(G)|}{2k},1\}$.
\end{thm}

Similarly, \cref{alpha} with $\alpha=2+\sqrt{2}$ (chosen to minimise $\frac{3\alpha(\alpha-1)}{\alpha-2}$) leads to the next result.

\begin{thm}
\label{ImprovedTreeProductSqrt}
For $k\in\NN$, every non-trivial graph $G$ with $\tw(G)\leq k-1$ has a $T$-partition of width at 
\[(1+\sqrt{2})k\,(3(1+\sqrt{2})\Delta(G) -1),\]
for some tree $T$ with $\Delta(T)\leq (3+3\sqrt{2})\Delta(G)-3(\sqrt{2}-1)$ and $|V(T)|\leq \max\{\tfrac{2|V(G)|}{(2+\sqrt{2}) k},1\}$. 
\end{thm}

{\fontsize{11pt}{12pt}\selectfont
\bibliographystyle{DavidNatbibStyle}
\bibliography{DavidBibliography}}
\end{document}